\documentclass[11pt,oneside]{article}
\usepackage{amsfonts}
\usepackage{amsmath}
\usepackage{amsthm}
\usepackage[mathscr]{euscript}
\usepackage{mathrsfs}
\usepackage{indentfirst}
\usepackage{xcolor}

\newtheorem{theo}{Theorem}[section]
\newtheorem{prop}{Proposition}[section]
\newtheorem{lem}{Lemma}[section]
\newtheorem{remark}{Remark}[section]

\newcommand{\be}{\begin{equation}}
\newcommand{\ee}{\end{equation}}
\newcommand\bes{\begin{eqnarray}} \newcommand\ees{\end{eqnarray}}
\newcommand{\bess}{\begin{eqnarray*}}
\newcommand{\eess}{\end{eqnarray*}}

\newcommand\kk{\left}
\newcommand\rr{\right}

\newcommand\dd{\displaystyle}

\newcommand\df{\dd\frac}

\newcommand\ty{T_{max}}

\newcommand\tyy{\tilde T_{max}}
\newcommand\lm{\lambda}

\newcommand\nm{\nonumber}
\newcommand\yy{\infty}
\newcommand\nn{\nabla}

\newcommand\pl{\partial}

\newcommand\ii{\int_\Omega}

\newcommand\oo{\Omega}

\newcommand\R{\mathbb{R}}

\newcommand\tr{\Delta}

\newcommand\ub{\bar u}
\newcommand\vb{\bar v}
\newcommand\wb{\bar w}

\newcommand\dv{\frac{d}{dt}}

\newcommand\hm{\mathcal{H}}
\newcommand\faa{{\rm for\ all}}

\newcommand\cd{\cdot}
\newcommand\tk{\tilde k}

\begin{document}
\setlength{\baselineskip}{16pt} \pagestyle{myheadings}

\begin{center}{\LARGE\bf Global existence and stabilization in a }\\[2mm]
{\LARGE\bf  forager-exploiter model with general logistic sources}\footnote{This work was partially supported by {\it Educational research projects for young and middle-aged teachers in Fujian (No. JAT200480)} and {\it Startup Foundation for Advanced Talents of Xiamen University of Technology (No. YKJ20019R)}.}\\[4mm]
 {\Large  Jianping Wang\footnote{Corresponding author. {\sl E-mail}: jianping0215@163.com}}\\[1mm]
{School of Applied Mathematics, Xiamen University of Technology, Xiamen, 361024, China}
\end{center}

\begin{quote}
\noindent{\bf Abstract.} We study a forager-exploiter model with generalized logistic sources in a smooth bounded domain with homogeneous Neumann boundary conditions. A new boundedness criterion is developed to prove the global existence and boundedness of the solution. Under some conditions on the logistic degradation rates, the classical solution exists globally and remain bounded in the high dimensions. Moreover, the large time behavior of the obtained solution is investigated in the case of the nutrient supply is a positive constant or has fast decaying property.


 \end{quote}

 \section{Introduction}
 \setcounter{equation}{0} {\setlength\arraycolsep{2pt}
We consider the forager-exploiter system with generalized logistic sources
\bes
 \left\{\begin{array}{lll}
 u_t=\tr u-\chi\nn\cd(u\nn w)+a_1u-b_1u^\alpha,&x\in\Omega,\ \ t>0,\\[1mm]
 v_t=\tr v-\xi\nn\cd(v\nn u)+a_2v-b_2v^\beta,&x\in\Omega,\ \ t>0,\\[1mm]
 w_t=\tr w-(u+v)w-\mu w+r,&x\in\Omega,\ \ t>0,\\[1mm]
 \frac{\partial u}{\partial\nu}=\frac{\partial v}{\partial\nu}=\frac{\partial w}{\partial\nu}=0,\ \ &x\in\partial\Omega,\ t>0,\\[1mm]
  u(x,0)=u_0(x),\ v(x,0)=v_0(x),\ w(x,0)=w_0(x),\ &x\in\Omega,
 \end{array}\right.\label{1.1}
 \ees
where $\oo$ is a bounded domain in $\R^N$ with smooth boundary $\pl\oo$, $\frac{\partial }{\partial\nu}$ denotes the homogeneous Neumann boundary condition, and $\chi,\xi,\mu$ are positive constants. We use $u,v$ and $w$ to denote population densities of the forager, exploiter and nutrient, respectively. This model include two basic settings. Firstly, the forager and exploiter pursuit nutrient as their common food. Secondly, except the random diffusions, the forager move towards the increasing nutrient gradient direction (corresponding to the first taxis dynamic: $-\nn\cd(u\nn w)$), while the exploiter follows the forager to find nutrients (corresponding to the second taxis mechanism: $-\nn\cd(v\nn u)$). It has been indicated in \cite{Tania2012,Winkler-M3AS2019} that, as a doubly cross-diffusive parabolic system, \eqref{1.1} possess more complex dynamics than the single-taxis model. Although problem \eqref{1.1} has attracted many attentions recently, only a few results are available. The mathematical results are summarized as follows:

{\bf Without kinetic source terms: $a_1=a_2=b_1=b_2=0$.} The classical solution exists globally and remain bounded in one space dimension (\cite{TaoW-2019-forager}). Whereas, in the high dimensions, smallness conditions on the initial data and production rate or weak taxis effects are required to ensure the global solvability of \eqref{1.1} (\cite{WW-M3AS2020}). This differs from the chemotaxis-consumption or prey taxis model whose global solvability in two dimensional case is well-known (\cite{JinW,Taow-jde2012}). Even for the generalized solutions, the global existence relies on some smallness conditions on the initial data $w_0$ and production rate of $w$. In \cite{cao-tao2021NARWA,Liu-nwrwa2019}, it is found that the full saturation or limited saturation can exclude blow up phenomenon. For the study of the forager-exploiter model with singular sensitivities, nonlinear resource consumption or proliferation, please refer to \cite{cao-M3AS2020,Liw-zamp2021,lz-zamp2020}.

{\bf With generalized logistic source: $a_1,a_2,b_1,b_2>0$.} In \cite{Black-M3AS2020}, the global generalized solutions are obtained in two dimensional setting provided $\alpha>\sqrt{2}+1$ and $\min\{\alpha,\beta\}>(\alpha+1)/(\alpha-1)$, and the eventual smoothness of the generalized solution after some waiting time are provided under some more restrict conditions. When $2\le \alpha<3$ and $\beta\ge3\alpha/(2\alpha-3)$ or $\alpha,\beta\ge3$, the global existence and boundedness of the classical solution in two dimensional setting are obtained in \cite{WW-M3AS2020}. Later on, the recent work \cite{mu-2020dcds} claims that $2\le \alpha<3,\beta\ge3$ suffices to ensure the global boundedness of the solution. It is noted that these results only hold in two dimensional setting, and the conditions for $\alpha,\beta$ seem too strong and may be far from optimal. Moreover, no results on the global solvability of \eqref{1.1} with $a_1,a_2,b_1,b_2>0$ in high dimensions ($N>2$) are available. This manuscript has two motivations. The first one is to establish the global solvability of \eqref{1.1} in the high dimensions ($N>2$) and relax the conditions in two dimensional case. The other one is to establish the large time behavior of the solutions obtained.

Before stating the main results, we give basic assumptions for the initial data and nutrient supply $r$. The initial data $u_0,v_0,w_0$ satisfy
   \[u_0,v_0,w_0\in W^{2,\infty}(\Omega), \ u_0,v_0,w_0\ge 0,\,\not\equiv 0\ \ {\rm on}\ \bar\Omega, \ \ \mbox{and} \ \ \frac{\partial u_0}{\partial\nu}=\frac{\partial v_0}{\partial\nu}=\frac{\partial w_0}{\partial\nu}=0\ \ {\rm on}\ \partial\Omega.\]
The production rate function $r$ is nonnegative and satisfies
\bes
r\in C^1(\bar\Omega\times[0,\infty))\cap L^\infty(\Omega\times(0,\infty)). \label{1.2}
\ees
We can hence denote
\bes
r_*:=\|r\|_{L^\yy(\oo\times(0,\yy))}.\label{1.3}
\ees

Throughout this article, we always set $a_1,a_2,b_1,b_2>0$. The first result concerns the global solvability of \eqref{1.1} in high dimensions (i.e., $N>2$).
\begin{theo}\label{t1.1}
Let $N>2$. Suppose that $\alpha,\beta>\frac{N}2+1$. Then the problem \eqref{1.1} admits a unique nonnegative and global solution $(u,v,w)\in(C^{2,1}(\bar\oo\times(0,\yy)))^3$ which is bounded in $\bar\oo\times(0,\yy)$.
\end{theo}

In two dimensional setting, we have
\begin{theo}\label{t1.2}
Let $N=2$ and $\alpha\ge2$ with $\beta>2$. Then there exists a solution $(u,v,w)\in(C^{2,1}(\bar\oo\times(0,\yy)))^3$ solving \eqref{1.1} uniquely and remain bounded in $\bar\oo\times(0,\yy)$.
\end{theo}
We remark that, Theorem \ref{t1.1} gives the first result of global solvability of forager-exploiter model with generalized logistic sources in high dimensions, and Theorem \ref{t1.2} improves the previous results in \cite{WW-M3AS2020,mu-2020dcds}. The proofs of Theorem \ref{t1.1} and Theorem \ref{t1.2} are based on a boundedness criterion in Proposition \ref{p3.4}. This criterion is different from the known one in \cite{BBTW}. We use the nutrient-taxis model as an example to see the difference of these two criteria. Dropping $v$ and its equation, system \eqref{1.1} becomes a nutrient-taxis model
\bes
 \left\{\begin{array}{lll}
 u_t=\tr u-\chi\nn\cd(u\nn w)+a_1u-b_1u^\alpha,&x\in\Omega,\ \ t>0,\\[1mm]
 w_t=\tr w-uw-\mu w+r,&x\in\Omega,\ \ t>0,\\[1mm]
 \frac{\partial u}{\partial\nu}=\frac{\partial w}{\partial\nu}=0,\ \ &x\in\partial\Omega,\ t>0,\\[1mm]
  u(x,0)=u_0(x),\ w(x,0)=w_0(x),\ &x\in\Omega.
 \end{array}\right.\label{1.4}
 \ees
The boundedness criterion in \cite[Lemma 3.2]{BBTW} says that, the local solution of system \eqref{1.4} exists globally and remain bounded provided the $L^p$ regularity of $u$ with $p>\frac{N}{2}$ is uniformly bounded. Whereas, our boundedness criterion says that the condition can be replaced by the space-time $L^p$ regularity of $u$ with $p>\frac{N}{2}+1$. 

For the Keller-Segel chemotaxis model, the classical logistic source can prevent blow-up in two space dimension (\cite{osaki-na2002,xiang-jde2015}). Is this still true for \eqref{1.1}? Unfortunately, the conclusion in Theorem \ref{t1.2} holds for $\alpha=2$ and $\beta>2$. For the case $\beta=2$, we have to leave this problem in the future.

We next investigate the large time behavior of the solution obtained in the above theorems and see how the taxis mechanisms and logistic damping rates affect the large time behavior. Clearly, the nutrient supply $r$ is critical for the large time behavior of $w$. We shall consider two cases:
\begin{itemize}
\item Fast decaying resupply of nutrient, i.e.,
\bes
\int_0^\yy\ii r(x,t){\rm d}x{\rm d}t<\yy. \label{5.1a}
\ees
\item $r$ is a positive constant.
\end{itemize}
We define
\bess
\ub=\kk(\frac{a_1}{b_1}\rr)^{\frac{1}{\alpha-1}},\ \ \vb=\kk(\frac{a_2}{b_2}\rr)^{\frac{1}{\beta-1}}
\eess
and
\bess
\wb=\frac{r}{\ub+\vb+\mu}\ \ {\rm if}\ r\ {\rm is\ a\ positive\ constant}.
\eess

In the case that $r$ has fast decaying property, without any additional restrictions on the parameters, we show the large time behavior.
\begin{theo}\label{t1.3}
Let $\alpha,\beta>\frac{N}2+1$ when $N>2$, and $\alpha\ge2$ with $\beta>2$ when $N=2$. Suppose that $r$ satisfies \eqref{1.2} and \eqref{5.1a}. Then
\bess
\|u(\cdot,t)-\ub\|_{L^\yy(\oo)}+\|v(\cdot,t)-\vb\|_{L^\yy(\oo)}+\|w(\cdot,t)\|_{L^\yy(\oo)}\rightarrow0\ \ as\ \ t\rightarrow\yy.
\eess
Furthermore, if $r$ satisfies \eqref{1.2} and
\bess
\ii r(\cdot,t)\le Ke^{-\delta t}\ \ for\ all\ t>0
\eess
for some $K,\delta>0$, then there exist $\lm,C>0$ such that
\bess
\|u(\cdot,t)-\ub\|_{L^\yy(\oo)}+\|v(\cdot,t)-\vb\|_{L^\yy(\oo)}+\|w(\cdot,t)\|_{L^\yy(\oo)}\le Ce^{-\lm t}\ \ for\ all\ t>0.
\eess
\end{theo}

If $r$ is a positive constant, the large time behavior reads as follows.
\begin{theo}\label{t1.4}
Let $\alpha,\beta>\frac{N}2+1$ when $N>2$, and $\alpha\ge2$ with $\beta>2$ when $N=2$. Assume that $r$ is a positive constant. Then there exist $\tilde\chi,\tilde\xi,\tilde b_1,\tilde b_2>0$ such that, once either
\bess\chi<\tilde\chi\eess
or
\bess b_1>\tilde b_1\ \  and\ \  b_2>\tilde b_2\eess
or
\bess b_1>\tilde b_1\ \  and \ \ \xi<\tilde \xi,\eess
one has
\bess
\|u(\cdot,t)-\ub\|_{L^\yy(\oo)}+\|v(\cdot,t)-\vb\|_{L^\yy(\oo)}+\|w(\cdot,t)-\wb\|_{L^\yy(\oo)}\le Ce^{-\lm t}\ \ for\ all\ t>0
\eess
for some $C,\lm>0$.
\end{theo}

The above theorem shows that, the solution will eventually reach a co-existence homogeneous steady state provided that $\chi$ is small enough, or the logistic damping rates $b_1$ and $b_2$ are sufficiently large, or the logistic damping rate $b_1$ is large and $\xi$ is small enough. It seems a little surprising that the large time behavior holds for small $\chi$, without any small condition for $\xi$.

\section{Existence and uniqueness of local solutions, some preliminaries}
\setcounter{equation}{0} {\setlength\arraycolsep{2pt}

The following lemma asserts the local-in-time existence of the classical solutions to \eqref{1.1}.

\begin{lem}\label{l2.1}\, Let $N\ge1$ and $\alpha,\beta>1$. Then there exist $\ty\in(0,\yy]$ and nonnegative functions
\bess
 u,v,w\in \bigcap_{q>n}C^0([0,\ty);W^{1,q}(\oo))\cap C^{2,1}(\bar\oo\times(0,\ty))
\eess
which solves \eqref{1.1} in $(0,\ty)$ in the classical sense. Moreover, if $\ty<\infty$, then
\[ \limsup_{t\to \ty}\kk(\|u(\cdot,t)\|_{W^{1,p}(\oo)}+\|v(\cdot,t)\|_{W^{1,p}(\oo)}+\|w(\cdot,t)\|_{W^{1,p}(\oo)}\rr)=\yy\ \ { for\ all}\ p>n.
 \]
\end{lem}
\begin{proof}
The local existence of the classical solution can be obtained by using Amann's theory, and the nonnegativity of the solution is deduced by the maximum principle (cf. \cite[Lemma 2.1]{TaoW-2019-forager}).
\end{proof}

We set from now on that $\tau=\min\{\frac{\ty}{2},1\}$ and $\tyy=\ty-\tau$. The following gives the $L^\yy$ bound for $w$ and some basic regularities for $u$ and $v$.
\begin{lem}\label{l2.2}
Let $N\ge1,\alpha,\beta>1$ and $M=\|w_0\|_{L^\yy(\oo)}+\frac{r_*}{\mu}$. Then
\bes
\|w(\cdot,t)\|_{L^\yy(\oo)}\le M\ \ for\ \ all\ t\in(0,\ty)\label{2.2}
\ees
and
\bes
\ii u\le M_1\ \ and\ \ \ii v\le M_2\ \ for\ \ all\ t\in(0,\ty) \label{2.3}
\ees
where $M_1=\ii u_0+|\oo|\kk(\frac{a_1}{b_1}\rr)^{\frac1{\alpha-1}}$ and $M_2=\ii v_0+|\oo|\kk(\frac{a_2}{b_2}\rr)^{\frac1{\beta-1}}$. Moreover, there exist $C=C(|\oo|,\ii u_0,a_1,\alpha)>0$ and $C'=C'(|\oo|,\ii v_0,a_2,\beta)>0$ such that
\bes
\int_t^{t+\tau}\ii u^{\alpha}\le Cb_1^{-1}(b_1^{-\frac1{\alpha-1}}+1)\ \ and\ \ \int_t^{t+\tau}\ii v^{\beta}\le C'b_2^{-1}(b_2^{-\frac1{\beta-1}}+1)\ \ for\ \ all\ t\in(0,\tyy). \label{2.4}
\ees
\end{lem}
\begin{proof}
The inequality \eqref{2.2} has been proven in \cite[Lemma 2.2]{TaoW-2019-forager}.

A simple use of H\"{o}lder's inequality provides $|\oo|^{1-\alpha}(\ii u)^\alpha\le \ii u^\alpha$. We then integrate the first equation in \eqref{1.1} over $\Omega$ to get
\bes
\dv\int_\Omega u&=&a_1\int_\Omega u-b_1\int_\Omega u^\alpha\nm\\
&\le&a_1\int_\Omega u-b_1|\oo|^{1-\alpha}(\ii u)^\alpha\ \ \ \ \ \ \faa\ t\in(0,\ty)\label{2.5a}
\ees
which implies
\bes
\int_\Omega u\le \ii u_0+|\oo|\kk(\frac{a_1}{b_1}\rr)^{\frac1{\alpha-1}}\ \ \ \ \ \ \faa\ t\in(0,\ty).\label{2.6a}
\ees
Integrating \eqref{2.5a} upon $(t,t+\tau)$ for $t\in(0,\tyy)$, and using \eqref{2.6a}, we have
\bess
\int_t^{t+\tau}\ii u^\alpha\le C_1b_1^{-1}(b_1^{-\frac1{\alpha-1}}+1)\ \ \ \ \ \ \faa\ t\in(0,\tyy)
\eess
for some $C_1=C_1(|\oo|,\ii u_0,a_1,\alpha)>0$. We thus obtain the regularity properties for $u$ in \eqref{2.3} and \eqref{2.4}. The statements for $v$ in \eqref{2.3} and \eqref{2.4} can be derived by the same way.
\end{proof}

The following lemma claims that, once we get the uniform boundedness of $u$ (resp. $v$), the estimation of $\nn u$ (resp. $\nn v$) relies on $\tr w,\nn w$ (resp. $\tr u,\nn u$).

\begin{lem}\label{l2.3} Let $N\ge1$ and $\alpha,\beta>1$.

{\rm(i)} Suppose that $\|u(\cdot,t)\|_{L^\yy(\oo)}\le k_1$ for all $t\in(0,\ty)$. Then for $p>1$, there exists $C=C(N,p,|\oo|,k_1,a_1,\chi)>0$ such that
\bes
\dv\ii|\nabla u|^{2p}+\ii|\nabla u|^{2p}\le C\kk(\ii|\nabla w|^{2(p+1)}+\ii|\Delta w|^{p+1}+1\rr)\ \ \ \ \ \ for\ all\ t\in(0,\ty).\label{2.5}
\ees

{\rm(ii)} Suppose that $\|v(\cdot,t)\|_{L^\yy(\oo)}\le k_2$ for all $t\in(0,\ty)$. Then for $p>1$, there exists $C_*=C_*(N,p,|\oo|,k_2,a_2,\xi)>0$ such that
\bes
\dv\ii|\nabla v|^{2p}+\ii|\nabla v|^{2p}\le C_*\kk(\ii|\nabla u|^{2(p+1)}+\ii|\Delta u|^{p+1}+1\rr)\ \ \ \ \ \ for\ all\ t\in(0,\ty).\label{2.6}
\ees
\end{lem}
\begin{proof}
By direct computations (cf. \cite{WW-M3AS2020}), we have
\bes
&&\df1{2p}\dv\ii|\nn u|^{2p}+\int_\Omega|\nabla u|^{2p}\nonumber\\
&=&\int_\Omega|\nabla u|^{2(p-1)}\nabla u\cdot\nabla u_t+\int_\Omega|\nabla u|^{2p}\nonumber\\
&=&\int_\Omega|\nabla u|^{2(p-1)}\nabla u\cdot\nabla(\Delta u-\chi\nabla\cdot(u\nabla w)+a_1u-b_1u^\alpha)+\int_\Omega|\nabla u|^{2p}\nonumber\\
&\le&\int_\Omega|\nabla u|^{2(p-1)}\nabla u\cdot\nabla\Delta u+\chi\int_\Omega\nabla\cdot(|\nabla u|^{2(p-1)}\nabla u)(\nabla\cdot(u\nabla w))+(a_1+1)\int_\Omega|\nabla u|^{2p}\nm\\
&=:&I(t)+J(t)+(a_1+1)\int_\Omega|\nabla u|^{2p}\ \ \ \ \ \ \faa\ t\in(0,\ty). \label{2.7a}
\ees
Following the derivation of \cite[(3.21)]{WW-M3AS2020}, one can find $C_1=C_1(N,p,|\oo|,k_1,\chi)>0$ fulfilling
\bes
&&\df1{2p}\dv\int_\Omega|\nabla u|^{2p}+\int_\Omega|\nabla u|^{2p}+\df1{16}\int_\Omega|\nabla u|^{2(p-1)}|D^2u|^2\nonumber\\
&\le&C_1\int_\Omega|\nabla w|^{2(p+1)}+C_1\int_\Omega|\Delta w|^{p+1}+\kk(C_1+a_1\rr)\int_\Omega|\nabla u|^{2p}\ \ \ \ \ \ \faa\ t\in(0,\ty).\label{2.8aa}
\ees
We use the third term in the left hand side of \eqref{2.8aa} to absorb $\kk(C_1+a_1\rr)\int_\Omega|\nabla u|^{2p}$ (cf. \cite[(3.22)]{WW-M3AS2020}), it then arrives at
\bess
&&\df1{2p}\dv\int_\Omega|\nabla u|^{2p}+\int_\Omega|\nabla u|^{2p}\nonumber\\
&\le&C_2\int_\Omega|\nabla w|^{2(p+1)}+C_2\int_\Omega|\Delta w|^{p+1}+C_2\ \ \ \ \ \ \faa\ t\in(0,\ty)
\eess
for some $C_2=C_2(N,p,|\oo|,k_1,a_1,\chi)>0$. The deduction of \eqref{2.6} is similar.
\end{proof}

We then construct a relationship between $\nn w$ and $u,v$.

\begin{lem}\label{l2.4}
Let $N\ge1$. For $p\ge1$, one can find $C=C(N,p,|\oo|,\|w_0\|_{L^\yy(\oo)},r_*,\mu)>0$ such that
\bes
\dv\ii|\nn w|^{2p}+\mu\ii|\nn w|^{2p}\le C\ii u^{p+1}+C\ii v^{p+1}+C\ \ \ \ \ for\ all\ t\in(0,\ty).\label{2.7}
\ees
\end{lem}
\begin{proof}
By straightforward computations, we have
\bes
&&\df1{2p}\dv\ii|\nn w|^{2p}+\mu\ii|\nn w|^{2p}\nonumber\\[0.5mm]
&=&\ii|\nn w|^{2(p-1)}\nn w\cdot\nn w_t+\mu\ii|\nn w|^{2p}\nonumber\\[0.5mm]
&=&\ii|\nn w|^{2(p-1)}\nn w\cdot\nn(\Delta w-(u+v)w-\mu w+r)+\mu\ii|\nn w|^{2p}\nonumber\\[0.5mm]
&=&\ii|\nn w|^{2(p-1)}\nn w\cdot\nn\Delta w-\ii|\nn w|^{2(p-1)}\nn w\cdot\nn(uw)-\ii|\nn w|^{2(p-1)}\nn w\cdot\nn(vw)\nm\\
&&+\ii|\nn w|^{2(p-1)}\nn w\cdot\nn r\nonumber\\[0.5mm]
&=:&I_1(t)+I_2(t)+I_3(t)+I_4(t)\ \ \ \ \ \ \faa\ t\in(0,\ty).\label{2.9}
\ees
In view of \cite[Lemma 2.2]{Lankeit-Wang} and \eqref{2.2}, there holds
\bes
\int_\Omega|\nabla w|^{2(p+1)}\le 2(N+4p^2)\|w\|_{L^\infty(\oo)}^2\int_\Omega|\nabla w|^{2(p-1)}|D^2w|^2\le \tk\int_\Omega|\nabla w|^{2(p-1)}|D^2w|^2\label{2.9a}
\ees
for all $t\in(0,\ty)$ with $\tk:=2(N+4p^2)M^2$ where $M$ is given by Lemma \ref{l2.2}.

Recalling \eqref{2.2}, similar to the derivation of \cite[(3.15)]{WW-M3AS2020} (or \cite[(4.14)]{WW-M3AS2020}), by adjusting some parameters, one can find $C_1=C_1(N,p,|\oo|,\|w_0\|_{L^\yy(\oo)},r_*,\mu)>0$ such that, for all $t\in(0,\ty)$,
\bes
I_1(t)&\le&-\frac{3}{4}\ii|\nn w|^{2(p-1)}|D^2w|^2-\df{p-1}{4}\ii|\nn w|^{2(p-2)}\kk|\nn|\nn w|^2\rr|^2+C_1.\label{2.8a}
\ees

Using integration by parts, Young's inequality, \eqref{2.9a} and the known inequality: $|\tr w|\le \sqrt{N}|D^2w|$, the second term on the right hand side of \eqref{2.9} can be estimated as:
\bes
I_2(t)&=&-\ii|\nn w|^{2(p-1)}\nn w\cdot\nn(uw)\nm\\
&=&(p-1)\ii uw|\nn w|^{2(p-2)}\nn|\nn w|^2\cdot\nn w+\ii uw|\nn w|^{2(p-1)}\tr w\nm\\
&\le&M(p-1)\ii u|\nn w|^{2(p-2)}\kk|\nn|\nn w|^2\rr||\nn w|+M\sqrt{N}\ii u|\nn w|^{2(p-1)}|D^2w|\nm\\
&\le&\frac{p-1}{16}\ii|\nn w|^{2(p-2)}\kk|\nn|\nn w|^2\rr|^2+\frac18\ii|\nn w|^{2(p-1)}|D^2w|^2\nm\\
&&+M^2(4(p-1)+2N)\ii u^2|\nn w|^{2(p-1)}\nm\\
&\le&\frac{p-1}{16}\ii|\nn w|^{2(p-2)}\kk|\nn|\nn w|^2\rr|^2+\frac18\ii|\nn w|^{2(p-1)}|D^2w|^2\nm\\
&&+\frac{1}{8\tk}\ii |\nn w|^{2(p+1)}+C_2\ii u^{p+1}\nm\\
&\le&\frac{p-1}{16}\ii|\nn w|^{2(p-2)}\kk|\nn|\nn w|^2\rr|^2+\frac14\ii|\nn w|^{2(p-1)}|D^2w|^2+C_2\ii u^{p+1}\label{2.10}
\ees
for some $C_2=C_2(N,p,\|w_0\|_{L^\yy(\oo)},r_*,\mu)>0$. Similarly, $I_3$ can be estimated as:
\bes
I_3(t)&=&-\ii|\nn w|^{2(p-1)}\nn w\cdot\nn(vw)\nm\\
&\le&\frac{p-1}{16}\ii|\nn w|^{2(p-2)}\kk|\nn|\nn w|^2\rr|^2+\frac14\ii|\nn w|^{2(p-1)}|D^2w|^2+C_2\ii v^{p+1}\label{2.11}
\ees
For the last term $I_4$, similar to the deduction of \eqref{2.10}, we find
\bes
I_4(t)&=&\ii|\nn w|^{2(p-1)}\nn w\cdot\nn r\nm\\
&=&-(p-1)\ii r|\nn w|^{2(p-2)}\nn|\nn w|^2\cdot\nn w-\ii r|\nn w|^{2(p-1)}\Delta w\nonumber\\[0.5mm]
&\le&r_*(p-1)\ii |\nn w|^{2(p-2)}\kk|\nn|\nn w|^2\rr||\nn w|-r_*\sqrt{N}\ii |\nn w|^{2(p-1)}|D^2 w|\nonumber\\[0.5mm]
&\le&\frac{p-1}{16}\ii|\nn w|^{2(p-2)}\kk|\nn|\nn w|^2\rr|^2+\frac18\ii|\nn w|^{2(p-1)}|D^2w|^2+r_*^2(4(p-1)+2N)\ii|\nn w|^{2(p-1)}\nm\\
&\le&\frac{p-1}{16}\ii|\nn w|^{2(p-2)}\kk|\nn|\nn w|^2\rr|^2+\frac18\ii|\nn w|^{2(p-1)}|D^2w|^2+\frac{1}{8\tk}\ii|\nn w|^{2(p+1)}+C_3\nm\\
&\le&\frac{p-1}{16}\ii|\nn w|^{2(p-2)}\kk|\nn|\nn w|^2\rr|^2+\frac14\ii|\nn w|^{2(p-1)}|D^2w|^2+C_3\label{2.12}
\ees
for all $t\in(0,\ty)$ for some $C_3=C_3(N,p,\|w_0\|_{L^\yy(\oo)},r_*,\mu)>0$. Plugging \eqref{2.8a}, \eqref{2.10}-\eqref{2.12} into \eqref{2.9}, one can find $C_4=C_4(N,p,|\oo|,\|w_0\|_{L^\yy(\oo)},r_*,\mu)>0$ such that
\bess
\df1{2p}\dv\ii|\nn w|^{2p}+\mu\ii|\nn w|^{2p}\le C_4\ii u^{p+1}+C_4\ii v^{p+1}+C_4\ \ \ \ \ \ \faa\ t\in(0,\ty).
\eess
This completes the proof.
\end{proof}

We collect an important lemma from \cite{WW-M3AS2020,Lou-Winkler,Black-M3AS2020}.
\begin{lem}\label{l2.5}
Let $N\ge1$, $T>0$ and $\theta=\min\{1,\frac{T}{2}\}$. Suppose that for some $p>1$ and $K,H>0$,
\bes
\int_t^{t+\theta}\ii |f|^p\le K\ \ and\ \ \int_t^{t+\theta}\ii |z|^p\le H\ \ \ \ \ for\ all\ t\in(0,T-\theta),\label{2.22c}
\ees
and $z\in C^{2,1}(\bar\oo\times(0,T))$ solves
\bess
 \left\{\begin{array}{lll}
 z_t=\tr z+f(x,t),&x\in\Omega,\ \ 0<t<T,\\[1mm]
 \frac{\partial z}{\partial\nu}=0,\ \ &x\in\partial\Omega,\ 0<t<T,\\[1mm]
  z(x,0)=z_0(x),\ &x\in\Omega,
 \end{array}\right.
 \eess
where $z_0\in W^{2,\yy}(\oo)$ with $z_0\ge 0$ and $\frac{\partial z_0}{\partial\nu}=0$ on $\pl\oo$. Then, there is $C>0$ fulfilling
\bes
\int_t^{t+\theta}\|z_t\|_{L^p(\oo)}^{p}+\int_t^{t+\theta}\|z\|_{W^{2,p}(\oo)}^{p}\le C(\|z_0\|_{W^{2,p}(\oo)}^{p}+K+H)\ \ \ \ \ for\ all\ t\in(0,T-\theta).\label{2.22d}
\ees
Especially, one can find $C_*>0$ such that
\bess
\int_t^{t+\theta}\ii|\tr z|^p\le C(\|z_0\|_{W^{2,p}(\oo)}^{p}+K+H)\ \ \ \ \ for\ all\ t\in(0,T-\theta).
\eess
\end{lem}
\begin{proof}
We used the ideas in \cite{Black-M3AS2020,Lou-Winkler} to show this lemma. The proof will be split into two cases.

{\bf Case I: $T\le2$}. It is easy to see that $\theta=T/2$, i.e., $T=2\theta$. Then, we use the maximal Sobolev regularity properties of the Neumann heat semigroup $(e^{t\Delta})_{t\ge0}$ (\cite{Giga-Sohr}) and the first inequality in \eqref{2.22c} to get $C_1>0$ fulfilling
\bes
&&\int_0^{2\theta}\|z_t\|_{L^p(\oo)}^{p}+\int_0^{2\theta}\|z\|_{W^{2,p}(\oo)}^{p}\nm\\
&\le&C_1\|z_0\|_{W^{2,p}(\oo)}^{p}+C_1\int_0^{2\theta}\|f\|_{L^p(\oo)}^p\nm\\
&\le& C_1\|z_0\|_{W^{2,p}(\oo)}^{p}+2C_1K.\label{2.22a}
\ees
Thus, we get \eqref{2.22d} directly.

{\bf Case II: $T>2$}. Clearly, we have $\theta=1$. Hence, \eqref{2.22d} can be rewritten as
\bes
\int_t^{t+1}\|z_t\|_{L^p(\oo)}^{p}+\int_t^{t+1}\|z\|_{W^{2,p}(\oo)}^{p}\le C(\|z_0\|_{W^{2,p}(\oo)}^{p}+K+H)\ \ \ \ \ \ \faa\ t\in(0,T-1). \label{2.22b}
\ees

Let $\rho\in C^\infty(\R)$ be an increasing function satisfying
\[0\le\rho\le1\ \ {\rm in}\ \R,\ \ \rho\equiv0\ \ {\rm in}\ (-\infty,0],\ \ \rho\equiv1\ \ {\rm in}\ (1,\infty).\]
It is easy to see that $\|\rho'\|_{L^{\yy}(\R)}=\|\rho'\|_{L^{\yy}([0,1])}$.
For arbitrary fixed $\sigma\in(0,T-2)$, we define $\rho_\sigma(t)=\rho(t-\sigma)$. Then, $\rho_\sigma\in[0,1]$ and $\|\rho_\sigma'\|_{L^{\yy}(\R)}=\|\rho'\|_{L^{\yy}(\R)}=\|\rho'\|_{L^{\yy}([0,1])}$. By direct computations, $\rho_\sigma z$ solves
\bess
 \left\{\begin{array}{lll}
 (\rho_\sigma z)_t=\tr (\rho_\sigma z)+\rho_\sigma' z+\rho_\sigma f(x,t),&x\in\Omega,\ \ t\in(\sigma,T),\\[1mm]
 \frac{\partial(\rho_\sigma z)}{\pl\nu}=0,\ \ &x\in\partial\Omega,\ t\in(\sigma,T),\\[1mm]
 (\rho_\sigma z)(x,\sigma)=0,\ &x\in\Omega.
 \end{array}\right.
 \eess
Making use of the maximal Sobolev regularity properties of the Neumann heat semigroup $(e^{t\Delta})_{t\ge0}$ (\cite{Giga-Sohr}), there exist $C_2>0$ such that
\bess
\int_\sigma^{\sigma+2}\| (\rho_\sigma z)_t\|_{L^p(\oo)}^{p}+\int_\sigma^{\sigma+2}\|\rho_\sigma z\|_{W^{2,p}(\oo)}^{p}
\le C_2\int_\sigma^{\sigma+2}\|\rho_\sigma' z+ \rho_\sigma f\|_{L^{p}(\oo)}^{p}\ \ \ \ \ \ \faa\ \sigma\in(0,T-2).
\eess
Thanks to \eqref{2.22c} and the boundedness properties of $\rho_\sigma$, we have from the above inequality that
\bess
&&\int_\sigma^{\sigma+2}\| (\rho_\sigma z)_t\|_{L^p(\oo)}^{p}+\int_\sigma^{\sigma+2}\|\rho_\sigma z\|_{W^{2,p}(\oo)}^{p}\nm\\
&\le& C_2\|\rho_\sigma'\|_{L^{\yy}(\R)}\int_\sigma^{\sigma+2} \|z\|_{L^p(\oo)}^p+ C_2\int_\sigma^{\sigma+2}\|f\|_{L^{p}(\oo)}^{p}\nm\\
&\le& C_3(K+H)\ \ \ \ \ \ \faa\ \sigma\in(0,T-2),
\eess
which combined with $\rho_\sigma=1$ in $(\sigma+1,\sigma+2)$ implies
\bess
\int_{\sigma+1}^{\sigma+2}\| z_t\|_{L^p(\oo)}^{p}+\int_{\sigma+1}^{\sigma+2}\|z\|_{W^{2,p}(\oo)}^{p}\le C_3(K+H)\ \ \ \ \ \ \faa\ \sigma\in(0,T-2),
\eess
i.e.,
\bess
\int_{t}^{t+1}\| z_t\|_{L^p(\oo)}^{p}+\int_{t}^{t+1}\|z\|_{W^{2,p}(\oo)}^{p}\le C_3(K+H)\ \ \ \ \ \ \faa\ t\in(1,T-1).
\eess
For $t\in(0,1]$, we have \eqref{2.22b} from \eqref{2.22a} with $\theta=1$. So, we get \eqref{2.22d} in the case of $\ty>2$. The proof is finished.
\end{proof}

Before ending this section, we recall from \cite{WW-M3AS2020} a generalized version of Gronwall's inequality.
\begin{lem}\label{l2.6}
Let $c,k>0$. Assume that for some $\hat T\in(0,\infty]$ and $\hat\tau=\min\{1,\frac{\hat T}{2}\}$, the nonnegative functions $y\in C([0,\hat T))\cap C^1((0,\hat T))$, $z\in L_{loc}^1([0,\hat T))$ and satisfy
 \bess
 &y'(t)+cy(t)\le z(t),\ \ \ t\in(0,\hat T),&\\[1mm]
 &\dd\int_t^{t+\hat\tau}z(s){\rm d}s\le k,\ \ \ t\in(0,\hat T-\hat\tau).&\nonumber
 \eess
Then
\bess
y(t)\le y(0)+2k+\frac{k}{c},\ \ \ t\in(0,\hat T).
\eess
\end{lem}

\section{A criteria governing boundedness of solutions}
\setcounter{equation}{0} {\setlength\arraycolsep{2pt}
It is known that, the uniform-in-time $L^\yy$ boundedness of $u$ can be ensured by \eqref{2.3} and $L^p$ regularity of $\nn w$ with $p>N$. And, Lemma \ref{l2.4} and Lemma \ref{l2.6} tell us that, we can use the space-time integral property of $u,v$ to estimate the $L^p$ regularity of $\nn w$. Hence, we obtain the uniform boundedness of $u$ provided suitable space-time regularity for $u$ and $v$. The upper bound for $u$ in the coming lemma is independent of $\chi,\xi,b_1,b_2$, which is critical for the construction of the global stability of the positive equilibrium.
\begin{lem}\label{l3.1}
Let $N\ge2$ and $\alpha,\beta>1$. Suppose that there exist $\bar p,\bar q>\frac{N}2+1$ and $K_1,K_2>0$ such that
\bes
\int_t^{t+\tau}\ii u^{\bar q}\le K_1\ \ and\ \ \int_t^{t+\tau}\ii v^{\bar p}\le K_2\ \ for\ \ all\ t\in(0,\tyy). \label{3.1}
\ees
Then, there exist $C>0, \theta>1$ and $\eta=:\min\{\bar p,\bar q\}>\frac{N}2+1$ independent of $\chi,\xi,b_1,b_2,K_1,K_2$ fulfilling
\bes
\|\nn w(\cdot,t)\|_{L^{2(\eta-1)}(\oo)}\le C(K_1+K_2+1)\ \ \ \ \ for\ all\ t\in(0,\ty) \label{3.2}
\ees
and
\bes
\|u(\cdot,t)\|_{L^\yy(\oo)}\le C(b_1^{-\frac{1}{\alpha-1}}+1)+C\chi^\theta M_1(K_1+K_2+1)^\theta\ \ \ \ \ for\ all\ t\in(0,\ty). \label{3.3}
\ees
\end{lem}
\begin{proof} We use $C_i$ to denote the general constants appeared in the proof, which are independent of $t$ and $\chi,\xi,b_1,b_2,K_1,K_2$. Let $\kappa:=2(\eta-1)$. Clearly, $\kappa>N$. We have from \eqref{3.1} that
\bess
\int_t^{t+\tau}\kk(\ii (u^\eta+v^\eta)+1\rr)\le C_1(K_1+K_2+1)\ \ \ \ \ \ \faa\ t\in(0,\tyy)
\eess
for some $C_1>0$. By Lemma \ref{l2.4}, there are $C_2,C_3>0$ such that
\bess
\dv\ii |\nn w|^{\kappa}+C_2\ii|\nn w|^\kappa\le C_3\kk(\ii (u^\eta+v^\eta)+1\rr)\ \ \ \ \ \ \faa\ t\in(0,\ty).
\eess
Then, in light of Lemma \ref{l2.6}, we get
\bes
\ii |\nn w|^{\kappa}&\le& \ii |\nn w_0|^{\kappa}+\kk(2+\frac{1}{C_2}\rr)C_1C_3(K_1+K_2+1)\nm\\
&\le&C_4(K_1+K_2+1)\ \ \ \ \ \ \faa\ t\in(0,\ty)\label{3.5c}
\ees
where $C_4=\ii |\nn w_0|^{\kappa}+\kk(2+\frac{1}{C_2}\rr)C_1C_3+1$. We thus get \eqref{3.2}. Since $C_4>1$ and $K_1+K_2+1>1$, we have from \eqref{3.5c} that
\bes
\|\nn w(\cdot,t)\|_{L^\kappa(\oo)}\le C_4(K_1+K_2+1)\ \ \ \ \ \ \faa\ t\in(0,\ty).\label{3.5}
\ees

Let us take $N<q<\kappa$ and set $\theta:=\frac{\kappa q}{\kappa-q}>N$. We denote
\bess
H(T)=\sup_{t\in(0,T)}\|u(\cdot,t)\|_\infty<\infty\ \  {\rm for}\ \  T\in(0,\ty).
\eess
On the basis of a variation-of-constants representation of $u$ and the known regularized properties of $(e^{t\Delta})_{t\ge0}$ (\cite{Fujie-I-W,M.W.2010}), one can find $\lm_1,\,C_5>0$ such that
\bes
\|u(\cdot,t)\|_\infty&\le& \|u_0\|_{L^\yy(\oo)}+\chi\int_0^t\|e^{(t-s)\Delta}\nabla\cdot(u\nabla w)\|_{L^\yy(\oo)}{\rm d}s+\int_0^t\|e^{(t-s)\Delta}(a_1u-b_1u^\alpha)\|_{L^\yy(\oo)}{\rm d}s\nonumber\\[0.5mm]
&\le&\|u_0\|_{L^\yy(\oo)}+\chi\int_0^t\|e^{(t-s)\Delta}\nabla\cdot(u\nabla w)\|_{L^\yy(\oo)}{\rm d}s+\int_0^t\|e^{(t-s)\Delta}(a_1u-b_1u^\alpha)_+\|_{L^\yy(\oo)}{\rm d}s\nonumber\\[0.5mm]
&\le&C_5+C_5\chi\int_0^t\kk(1+(t-s)^{-\frac12-\frac{N}{2q}}\rr)e^{-\lm_1(t-s)}\|u\nabla w\|_{L^q(\oo)}{\rm d}s\nm\\
&&+\int_0^t\|e^{(t-s)\Delta}(a_1u-b_1u^\alpha)_+\|_{L^\yy(\oo)}{\rm d}s\ \ \ \ \ \ \faa\ t\in(0,T).\label{3.4}
\ees
Making use of H\"{o}lder's inequality, \eqref{3.5} and the first inequality in \eqref{2.3}, we infer that, for all $t\in(0,T)$,
\bes
\|u\nabla w\|_{L^q(\oo)}\le\|u\|_{\theta}\|\nabla w\|_{L^{\kappa}(\oo)}
&=&\kk(\int_\Omega u^{\theta}\rr)^{\frac1{\theta}}\|\nabla w\|_{L^{\kappa}(\oo)}\nm\\
&\le& \kk(\int_\Omega u\rr)^{\frac1{\theta}}H^{\frac{\theta-1}{\theta}}(T)\|\nabla w\|_{L^{\kappa}(\oo)}\nm\\
&\le&C_4M_1^{\frac1{\theta}}(K_1+K_2+1)H^{\frac{\theta-1}{\theta}}(T)\label{3.5aa}
\ees
Letting $f(u)=a_1u-b_1u^\alpha$, there exists $C_6>0$ such that
$$(a_1u-b_1u^\alpha)_+\le f\kk(\kk(\frac{a_1}{b_1\alpha}\rr)^{\frac1{\alpha-1}}\rr)\le C_6 b_1^{-\frac{1}{\alpha-1}}.$$
In conjunction with the known regularized properties of $(e^{t\Delta})_{t\ge0}$ (\cite{M.W.2010}), this shows that
\bes
&&\int_0^t\|e^{(t-s)\Delta}(a_1u-b_1u^\alpha)_+\|_{L^{\yy }(\oo)}{\rm d}s\nm\\
&\le& C_7\int_0^te^{-\lm_1(t-s)}\|(a_1u-b_1u^\alpha)_+\|_{L^{\yy}(\oo)}{\rm d}s\nm\\
&\le& C_8 b_1^{-\frac{1}{\alpha-1}}\ \ \ \ \ \ \faa\ t\in(0,\ty)\label{3.5bb}
\ees
for some $C_7,C_8>0$. Inserting \eqref{3.5aa} and \eqref{3.5bb} into \eqref{3.4} yields that, for some $C_9>0$,
\bess
\|u(\cdot,t)\|_\infty&\le&C_9(b_1^{-\frac{1}{\alpha-1}}+1)\nm\\
&&+C_9\chi M_1^{\frac1{\theta}}(K_1+K_2+1)H^{\frac{\theta-1}{\theta}}(T)\int_0^t\kk(1+(t-s)^{-\frac12-\frac{N}{2q}}\rr)e^{-\lm_1(t-s)}{\rm d}s,\nonumber\\[0.5mm]
&\le&C_9(b_1^{-\frac{1}{\alpha-1}}+1)+C_{10}\chi M_1^{\frac1{\theta}}(K_1+K_2+1)H^{\frac{\theta-1}{\theta}}(T)\ \ \ \faa\ t\in(0,T),
\eess
which implies
\[H(T)\le C_9(b_1^{-\frac{1}{\alpha-1}}+1)+C_{10}\chi M_1^{\frac1{\theta}}(K_1+K_2+1)H^{\frac{\theta-1}{\theta}}(T)\ \ \ \faa\ T\in(0,\ty).\]
Hence, thanks to Young's inequality, we have
\bess
H(T)\le C_{11}(b_1^{-\frac{1}{\alpha-1}}+1)+C_{11}\chi^\theta M_1(K_1+K_2+1)^\theta
\eess
for all $ T\in(0,\ty)$. This combined with the definition of $H(T)$ finishes the proof.
\end{proof}

We proceed to derive the uniform-in-time $L^\yy$ boundedness of $v$. Similar to the situation of $u$, the boundedness of $v$ relies on $\nn u$. Since we have obtained the $L^\yy$ boundedness of $u$, Lemma \ref{2.4}(i) can be applied to estimate $\nn u$. To achieve this, we need to establish the space-time $L^p$ integral property of $\tr w$ which is ensured by the Sobolev maximal regularity asserted in Lemma \ref{l2.5}.

\begin{lem}\label{l3.2}
Let $N\ge2$ and $\alpha,\beta>1$. Assume that for some $\bar p,\bar q>\frac{N}{2}+1$, there exist $K_1,K_2>0$ fulfilling
\bess
\int_t^{t+\tau}\ii u^{\bar q}\le K_1\ \ \ \ \ for\ all\ t\in(0,\tyy)
\eess
and
\bes
\int_t^{t+\tau}\ii v^{\bar p}\le K_2\ \ \ \ \ for\ all\ t\in(0,\tyy). \label{3.6}
\ees
Then there exists $C>0$ such that
\bes
\ii|\nn u|^{2(\bar p-1)}\le C\ \ \ \ \ \ \faa\ t\in(0,\ty),\label{3.6a}
\ees
and
\bes
\|v(\cdot,t)\|_{L^\yy(\oo)}\le C\ \ \ \ \ for\ all\ t\in(0,\ty).\label{3.17}
\ees
\end{lem}
\begin{proof}
Let $f(x,t):=-(u+v)w-\mu w+r$. By \eqref{1.3}, \eqref{2.2}, \eqref{3.3} and \eqref{3.6}, there exists $C_1>0$ such that
 \bes
\int_t^{t+\tau}\ii |f|^{\bar p}\le C_1\ \ \ \ \ \ \faa\ t\in(0,\tyy).\label{3.7bb}
 \ees
It follows from \eqref{1.1} that $w$ satisfies
\bes
 \left\{\begin{array}{lll}
 w_t=\tr w+f(x,t),&x\in\Omega,\ \ t>0,\\[1mm]
 \frac{\partial w}{\partial\nu}=0,\ \ &x\in\partial\Omega,\ t>0,\\[1mm]
 w(x,0)=w_0(x),\ &x\in\Omega.
 \end{array}\right.\label{3.6c}
 \ees
Thanks to \eqref{2.2} and \eqref{3.7bb}, we apply Lemma \ref{l2.5} to \eqref{3.6c} to find $C_2>0$ such that
\bes
\int_t^{t+\tau}\ii|\tr w|^{\bar p}\le C_2 \ \ \ \ \ for\ all\ t\in(0,\tyy).\label{3.7}
\ees
It follows from \eqref{3.3} and Lemma \ref{l2.3}(i) that
\bes
\dv\ii|\nn u|^{2(\bar p-1)}+\ii|\nn u|^{2(\bar p-1)}\le C_3\kk(\ii|\nn w|^{2\bar p}+\ii|\Delta w|^{\bar p}+1\rr)\ \ \ \ \faa\ t\in(0,\ty),\label{3.8}
\ees
where $C_3>0$. In view of the Gagliardo-Nirenberg inequality and \eqref{2.2}, there are $C_4,C_5>0$ such that
\bess
\ii|\nn w|^{2\bar p}&=&\|\nn w\|_{L^{2\bar p}(\oo)}^{2\bar p}\nm\\
&\le& C_4\|\tr w\|_{L^{\bar p}(\oo)}^{\bar p}\|w\|_{L^\yy(\oo)}^{\bar p}+C_4\|w\|_{L^\yy(\oo)}^{2\bar p}\nm\\
&\le& C_5\ii |\tr w|^{\bar p}+C_5.
\eess
Inserting this into \eqref{3.8} yields $C_6>0$ fulfilling
\bes
\dv\ii|\nn u|^{2(\bar p-1)}+\ii|\nn u|^{2(\bar p-1)}\le C_6\ii|\Delta w|^{\bar p}+C_6\ \ \ \ \ \ \faa\ t\in(0,\ty).\label{3.9}
\ees
By using \eqref{3.9}, \eqref{3.7} and Lemma \ref{l2.6}, there exists $C_7>0$ such that
\bes
\ii|\nn u|^{2(\bar p-1)}\le C_7\ \ \ \ \ \ \faa\ t\in(0,\ty).\label{3.10}
\ees
We thus obtain \eqref{3.6a}.

It follows from $\bar p>\frac{N}2+1$ that $\hat p:=2(\bar p-1)>N$. By \eqref{3.10}, there is $C_8>0$ such that
\bes
\ii|\nn u|^{\hat p}\le C_8\ \ \ \ \ \ \faa\ t\in(0,\ty).\label{3.18}
\ees
By standard arguments paralleled to Lemma \ref{l3.1}, one can deduce \eqref{3.17} by using \eqref{3.18} and \eqref{2.3}.
\end{proof}

The following provides a criterion for the global existence and  boundedness of the solution.
\begin{prop}\label{p3.4}
Let $N\ge2$ and $\alpha,\beta>1$. Suppose that there exist $\bar p,\bar q>\frac{N}2+1$ and $K_1,K_2>0$ fulfilling
\bess
\int_t^{t+\tau}\ii u^{\bar q}\le K_1\ \ \ \ \ for\ all\ t\in(0,\tyy)
\eess
and
\bess
\int_t^{t+\tau}\ii v^{\bar p}\le K_2\ \ \ \ \ for\ all\ t\in(0,\tyy).
\eess
Then $\ty=\yy$, and there exist $\theta\in(0,1)$ and $C>0$ fulfilling
\bes
&&\|u\|_{C^{2+\theta,1+\frac{\theta}2}(\bar\oo\times[t,t+1])}+\|v\|_{C^{2+\theta,1+\frac{\theta}2}(\bar\oo\times[t,t+1])}\nm\\
&&\quad\quad+\|w\|_{C^{2+\theta,1+\frac{\theta}2}(\bar\oo\times[t,t+1])}\le C\ \ \ \ \ for\ all\ t\in(0,\yy).\label{3.21a}
\ees
\end{prop}

\begin{proof}
Let us take $p>\frac{N}2$ throughout this proof. Thanks to \eqref{3.3} and \eqref{3.17}, we can use Lemma \ref{l2.4} and Gronwall's inequality to find $C_1>0$ fulfilling
\bes
\ii|\nn w|^{2(p+2)}\le C_1\ \ \ \ \ \ \faa\ t\in(0,\ty).\label{3.17a}
\ees
Again by \eqref{3.3} and \eqref{3.17}, similar to the derivation of \eqref{3.7}, there is $C_2>0$ such that
\bes
\int_t^{t+\tau}\ii|\tr w|^{p+2}\le C_2 \ \ \ \ \ for\ all\ t\in(0,\tyy).\label{3.22}
\ees
Based on \eqref{3.3}, \eqref{3.17a} and \eqref{3.22}, we apply Lemma \ref{l2.3}(i) and Lemma \ref{l2.6} to get $C_3>0$ fulfilling
\bes
\ii|\nn u|^{2(p+1)}\le C_3\ \ \ \ \ \ \faa\ t\in(0,\tyy).\label{3.23}
\ees
Rewriting the equation of $u$ in \eqref{1.1} as
\bess
 \left\{\begin{array}{lll}
 u_t=\Delta u+F(x,t),&x\in\Omega,\ \ t\in(0,T_m),\\[1mm]
 \frac{\partial u}{\pl\nu}=0,\ \ &x\in\partial\Omega, \ \ t\in(0,T_m),\\[1mm]
 u(x,0)=u_0, &x\in\Omega,
 \end{array}\right.
\eess
where $F(x,t)=-\chi(\nn u\cdot\nn w+u\Delta w)+a_1u-b_1u^\alpha$. Making use of Young's inequality, \eqref{3.3}, \eqref{3.17a}, \eqref{3.22} and \eqref{3.23}, there exists $C_4>0$ fulfilling
\bess
\int_t^{t+\tau}\!\!\ii |F(x,s)|^{p+1}{\rm d}x{\rm d}s\le C_4\ \ \ \ \ \ \faa\ t\in(0,\tyy).
\eess
This combined with \eqref{3.3} enables us to apply Lemma \ref{l2.5} to get $C_5>0$ such that
\bes
\int_t^{t+\tau}\!\!\ii|\Delta u|^{p+1}\le C_5\ \ \ \ \ \ \faa\ t\in(0,\tyy).\label{3.24}
\ees

With \eqref{3.17} at hand, we use Lemma \ref{l2.3}(ii) to get $C_6>0$ such that
\bess
\dv\ii|\nn v|^{2p}+\ii|\nn v|^{2p}\le C_6\kk(\ii|\nn u|^{2(p+1)}+\ii|\Delta u|^{p+1}+1\rr)\ \ \ \ \ \ \faa\ t\in(0,\ty).
\eess
Then, by \eqref{3.23}, \eqref{3.24} and Lemma \ref{l2.6}, we have
\bes
\ii|\nn v(\cdot,t)|^{2p}\le C_7\ \ \ \ \ \ \faa\ t\in(0,\ty) \label{3.25}
\ees
for $C_7>0$.

From \eqref{2.2}, \eqref{3.3}, \eqref{3.17}, \eqref{3.2}, \eqref{3.23} and \eqref{3.25}, one can find $q>N$ and $C_8>0$ such that
\bes
\|u(\cdot,t)\|_{W^{1,q}(\Omega)}+\|v(\cdot,t)\|_{W^{1,q}(\Omega)}+\|w(\cdot,t)\|_{W^{1,q}(\Omega)}\le C_8\ \ \ \ \ \ \faa\ t\in(0,\ty).\label{3.26}
\ees
This deduces that $\ty=\infty$ due to Lemma \ref{l2.1}. Based on \eqref{3.26}, the regularities in \eqref{3.21a} can be derived by a standard reasoning involving the known parabolic regularity theory in \cite{L-S-Y1968} (cf. \cite[Theorem 2.1]{WW-JDDE2020}).
\end{proof}

\begin{remark}\label{r3.1}
By some minor changes, it can be shown that the criterion in Proposition \ref{p3.4} also holds for the cases that (i) $a_1=b_1=0$ and $b_2>0$ with $\beta>1$, (ii) $a_2=b_2=0$ and $b_1>0$ and $\alpha>1$, (iii) $a_i=b_i=0$ for $i=1,2$.
\end{remark}

\section{Boundedness for the forager-exploiter model with logistic sources}
\setcounter{equation}{0} {\setlength\arraycolsep{2pt}
\begin{proof}[\rm\bf Proof of Theorem \ref{t1.1}]
Since $\alpha,\beta>\frac{N}{2}+1$, the conditions of Proposition \ref{p3.4} are satisfied according to \eqref{2.4}. Hence, we can get Theorem \ref{t1.1} from Proposition \ref{p3.4}.
\end{proof}

The following lemma shows that one can improve the space-time regularity of $u$ provided $\alpha\ge2$ and $\beta\ge2$ in two dimensional setting. We also establish an explicit upper bound so that it can be used in the stability arguments.
\begin{lem}\label{l4.1}
Let $N=2$. Suppose that $\alpha\ge2,\ \beta\ge2$. Then there exists $C>0$ independent of $\chi,\xi,b_1,b_2$ such that
\bes
\int_t^{t+\tau}\ii u^3\le C\kk(M_1\hm_1(\hm_1+\hm_2)e^{C(\chi^2+M_1^4+1)(\hm_1+\hm_2)}+M_1^3\rr)\ \ \ \ \ for\ all\ t\in(0,\tyy),\label{4.1}
\ees
where $\hm_1=b_1^{-1}(b_1^{-\frac1{\alpha-1}}+1)+1$ and $\hm_2=b_2^{-1}(b_2^{-\frac1{\beta-1}}+1)+1$.
\end{lem}
\begin{proof}
The derivation of \eqref{4.1} follows the ideas of \cite[Lemma 2.5]{baiw-2016} and \cite[Lemma 4.1]{WW-M3AS2020}. However, to see how $\chi,\xi,b_1,b_2$ influence the upper bound of $\int_t^{t+\tau}\ii u^3$, delicate analysis will be processed in the following. Keeping in mind that $\hm_1,\hm_2>1$. Let $f(x,t)=-(u+v)w-\mu w+r$. For the simplicity, we set
\[\Sigma:=\kk\{|\oo|,\ii u_0,\ii u_0^2,\ii v_0,\|w_0\|_{W^{2,\yy}(\oo)},a_1,a_2,\alpha,\beta,r_*,\mu\rr\}\]
By \eqref{2.4} with $\alpha,\beta\ge2$, there exists $C_1=C_1(|\oo|,\ii u_0,\ii v_0,a_1,a_2,\alpha,\beta)>0$ such that
\bes
\int_t^{t+\tau}\ii u^2\le C_1\hm_1\ \ and\ \ \int_t^{t+\tau}\ii v^2\le C_1\hm_2\ \ \ \ \ \ \faa\ t\in(0,\tyy). \label{4.2}
\ees

We deduce from \eqref{1.3}, \eqref{2.2} and \eqref{4.2} that
\[\int_t^{t+\tau}\!\!\int_\Omega |f|^2\le C_2(\hm_1+\hm_2),\ \ t\in(0,\tyy)\]
for some $C_2=C_2(\Sigma)>0$. Recalling \eqref{2.2}, an application of Lemma \ref{l2.5} yields $C_3=C_3(\Sigma)>0$ such that
\bes
\int_t^{t+\tau}\!\!\int_\Omega |\Delta w|^2\le C_3(\hm_1+\hm_2)\ \ \ \ \ \ \faa\ t\in(0,\tyy).\label{4.5}
\ees

Testing the first equation of \eqref{1.1} by $u$ and using \eqref{2.3}, the Gagliardo-Nirenberg inequality and Young's inequality, one can find $C_4,C_5>0$ depending upon $\Sigma$ fulfilling (cf. \cite[Lemma 4.1]{WW-M3AS2020})
\bess
&&\dv\ii u^2+2\ii|\nabla u|^2\nm\\
&\le&\chi C_4(\|\nabla u\|_2\|u\|_2+ M_1^2)\|\tr w\|_2+2a_1\ii u^2\nonumber\\[0.5mm]
&\le& \|\nn u\|^2_2+C_5(\chi^2+M_1^4+1)\big(\|u\|^2_2\|\tr w\|_2^2+\|u\|^2_2+\|\tr w\|_2^2+1\big)\ \ \ \ \ \ \faa\ t\in(0,\ty),
\eess
i.e.,
\bes
 z'(t)+\ii|\nn u|^2\le C_5(\chi^2+M_1^4+1)z(t)h(t)\ \ \ \ \ \ \faa\ t\in(0,\ty),\label{4.6}
\ees
where
 \[z(t)=\int_\Omega |u(\cdot,t)|^2+1, \ \ h(t)=\int_\Omega|\Delta w(\cdot,t)|^2+1.\]

We next claim that, there is $C_*=C_*(\Sigma)>0$ such that
\bes
z(t)=\int_\Omega |u(\cdot,t)|^2+1\le C_*\hm_1e^{C_*(\chi^2+M_1^4+1)(\hm_1+\hm_2)}\ \ \ \ \ \ \faa\ t\in(0,\ty).\label{4.7}
\ees
For any $0\le \tilde t\le t<\ty$, we have from \eqref{4.6} that
\bes
z(t)\le z(\tilde t)e^{C_5(\chi^2+M_1^4+1)\int_{\tilde t}^t h(s){\rm d}s}.\label{4.6a}
\ees
By \eqref{4.6a} and \eqref{4.5}, there is $C_6=C_6(\Sigma)>0$ such that
\bes
z(t)\le z(0)e^{C_5(\chi^2+M_1^4+1)\int_0^t h(s){\rm d}s}\le C_6e^{C_6(\chi^2+M_1^4+1)(\hm_1+\hm_2)}\ \ {\rm for}\ t\in[0,\tau].\label{4.8}
\ees

If $\ty<2$, then $\ty=2\tau$ by the definition of $\tau$. Thanks to \eqref{4.6a}, \eqref{4.5} and \eqref{4.8}, one can find $C_7=C_7(\Sigma)>0$ fulfilling
\bess
z(t)\le z(\tau)e^{C_5(\chi^2+M_1^4+1)\int_{\tau}^t h(s){\rm d}s}\le C_7e^{C_7(\chi^2+M_1^4+1)(\hm_1+\hm_2)}\ \ {\rm for}\ t\in(\tau,\ty).
\eess
This combined with \eqref{4.8} gives $C_8=C_8(\Sigma)>0$ such that
\bes
z(t)=\int_\Omega |u(\cdot,t)|^2+1\le C_8e^{C_8(\chi^2+M_1^4+1)(\hm_1+\hm_2)}\ \ \ \ \ \ \faa\ t\in(0,\ty).\label{4.6b}
\ees
We thus obtain \eqref{4.7} due to $\hm_1>1$.

When $\ty\ge2$, we have $\tau=1$. For $t\in(1,\ty)$, by the known mean value theorem and \eqref{4.2}, there exists $t_0\in[t-1,t]$ such that
\bess
z(t_0)=\int_\Omega u^2(\cdot,t_0)+1\le C_1\hm_1+1\le (C_1+1)\hm_1
\eess
Making use of \eqref{4.5}, we have
\bess
 \int_{t_0}^t h(s){\rm d}s=\int_{t_0}^t\kk(\int_\Omega|\Delta w|^2{\rm d}x+1\rr){\rm d}s
&\le&\int_{t-1}^t\kk(\int_\Omega|\Delta w|^2{\rm d}x+1\rr){\rm d}s\\
&\le& C_3(\hm_1+\hm_2)+1\le (C_3+1)(\hm_1+\hm_2).
\eess
Hence, we have from \eqref{4.6a} that
\bess
z(t)\le z(t_0)e^{C_5(\chi^2+M_1^4+1)\int_{t_0}^t h(s)ds}\le C_9\hm_1e^{C_9(\chi^2+M_1^4+1)(\hm_1+\hm_2)}\ \ \ \ \ \ \faa\ 1<t<\ty.
\eess
for some $C_9=C_9(\Sigma)>0$. In conjunction with \eqref{4.8}, this shows \eqref{4.7}.

Inserting \eqref{4.7} into \eqref{4.6} yields
\bess
\dv\int_\Omega u^2+\int_\Omega|\nabla u|^2\le C_*\hm_1e^{C_*(\chi^2+M_1^4+1)(\hm_1+\hm_2)}\kk(\int_\Omega|\Delta w|^2+1\rr)\ \ \ \ \ \ \faa\ t\in(0,\ty),
\eess
which by an integration upon $(t,t+\tau)$ for $t\in(0,\ty-\tau)$ implies that, for all $t\in(0,\tyy)$,
\bess
\int_t^{t+\tau}\ii|\nn u|^2\le\ii u^2(\cdot,t)+C_*\hm_1e^{C_*(\chi^2+M_1^4+1)(\hm_1+\hm_2)}\kk(\int_t^{t+\tau}\ii|\Delta w|^2+1\rr).
\eess
By means of \eqref{4.5} and \eqref{4.7} again, one can find $C_{10}=C_{10}(\Sigma)>0$ such that
\bes
\int_t^{t+\tau}\!\!\int_\Omega|\nabla u|^2\le C_{10}\hm_1(\hm_1+\hm_2)e^{C_{10}(\chi^2+M_1^4+1)(\hm_1+\hm_2)}\ \ \ \ \ \ \faa\ t\in(0,\ty).\label{4.9}
\ees

By the Gagliardo-Nirenberg inequality, we obtain from \eqref{2.3} that
\bes
\|u\|_{L^3(\oo)}^3&\le& C_{11}\kk(\|\nn u\|_{L^2(\oo)}^2\|u\|_{L^1(\oo)}+\|u\|_{L^1(\oo)}^3\rr)\nm\\
&\le&C_{11}\kk(M_1\|\nn u\|_{L^2(\oo)}^2+M_1^3\rr)\ \ \ \ \ \ \faa\ t\in(0,\ty)\nm
\ees
for some $C_{11}=C_{11}(\Sigma)>0$. This combined with \eqref{4.9} provides
\bes
\int_t^{t+\tau}\ii u^3\le C_{11}\kk(C_{10}M_1\hm_1(\hm_1+\hm_2)e^{C_{10}(\chi^2+M_1^4+1)(\hm_1+\hm_2)}+M_1^3\rr)\ \ \ \ \ \ \faa\ t\in(0,\tyy).\nm
\ees
This shows \eqref{4.1}. The proof is end.
\end{proof}

\begin{proof}[\rm\bf Proof of Theorem \ref{t1.2}]

From \eqref{4.1} and the second inequality in \eqref{2.4} with $\beta>2$, it is easy to see that the conditions of Proposition \ref{p3.4} are satisfied for $N=2$. Then, we get Theorem \ref{t1.2} from Proposition \ref{p3.4}.
\end{proof}

\section{Large time behavior of the solution}
\setcounter{equation}{0} {\setlength\arraycolsep{2pt}

The following explicit $L^\yy$-boundedness of $u$ plays an important role in the stability analysis in the case that $r$ is a positive constant.
\begin{lem}\label{l5.0}
{\rm(i)} Let $N>2$ and $\alpha,\beta>\frac{N}2+1$. Then there exist $C>0$ and $\theta>1$ independent of $\chi,\xi,b_1,b_2$ such that
\bess
M_u:=C(\chi^\theta+1)(b_1^{-\frac{1}{\alpha-1}}+1)\kk(b_1^{-1}(b_1^{-\frac{1}{\alpha-1}}+1)+b_2^{-1}(b_2^{-\frac{1}{\beta-1}}+1)+1\rr)^\theta
\eess
satisfying
\bes
\|u(\cdot,t)\|_{L^\yy(\oo)}\le M_u\ \ \ \ \ \ for\ all\ t\in(0,\yy).\label{5.1}
\ees

{\rm(ii)} Let $N=2$ and $\alpha\ge2$ with $\beta>2$. Then there exist $C>0$ and $\theta>1$ independent of $\chi,\xi,b_1,b_2$ such that
\bess
M_u:=C(b_1^{-\frac1{\alpha-1}}+1)+C\chi^\theta M_1\kk(M_1\hm_1(\hm_1+\hm_2)e^{C(\chi^2+M_1^4+1)(\hm_1+\hm_2)}+M_1^3+\hm_2\rr)^\theta
\eess
fulfilling \eqref{5.1}, where $\hm_1=b_1^{-1}(b_1^{-\frac1{\alpha-1}}+1)+1$ and $\hm_2=b_2^{-1}(b_2^{-\frac1{\beta-1}}+1)+1$.
\end{lem}
\begin{proof}
For $N>2$, we combine \eqref{2.3}, \eqref{2.4} with \eqref{3.3} to infer that
\bess
\|u(\cdot,t)\|_{L^\yy(\oo)}\le C(\chi^\theta+1)(b_1^{-\frac{1}{\alpha-1}}+1)\kk(b_1^{-1}(b_1^{-\frac{1}{\alpha-1}}+1)+b_2^{-1}(b_2^{-\frac{1}{\beta-1}}+1)+1\rr)^\theta
\eess
for some $C>0,\theta>1$. This shows \eqref{5.1}.

For $N=2$, we use \eqref{2.3}, \eqref{4.1}, the second inequality in \eqref{2.4} and \eqref{3.3} to get \eqref{5.1}.
\end{proof}

The following lemma provides an inequality which does not require any additional restrictions for $r$.
\begin{lem}\label{l5.1}
Let $N\ge2$ and
\bes
\alpha,\ \beta>\frac{N}{2}\ when\ N>2;\ \  \alpha\ge2,\ \beta>2\ for\ N=2. \label{5.6}
\ees
Then,
\bes
&&\dv\kk\{\ii\kk(u-\ub-\ub\ln\frac{u}{\ub}\rr)+\frac{\ub}{\xi^2\vb M_u^2}\ii\kk(v-\vb-\vb\ln\frac{v}{\vb}\rr)\rr\}\nm\\
&\le&\frac{\chi^2\ub}{2}\ii|\nn w|^2-b_1\ub^{\alpha-2}\ii(u-\ub)^2-\frac{b_2\ub\vb^{\beta-3}}{\xi^2M_u^2}\ii(v-\vb)^2\ \ \ \ for\ all\ t\in(0,\yy).\label{5.8}
\ees
\end{lem}
\begin{proof}
By standard calculations,
\bes
&&\dv\ii\kk(u-\ub-\ub\ln\frac{u}{\ub}\rr)\nm\\
&=&-\ub\ii\frac{|\nn u|^2}{u^2}-\chi\ub\ii\frac{\nn u\cdot\nn w}{u}-b_1\ii(u-\ub)(u^{\alpha-1}-\ub^{\alpha-1})\nm\\
&\le&-\frac{\ub}{2}\ii\frac{|\nn u|^2}{u^2}+\frac{\chi^2\ub}{2}\ii|\nn w|^2-b_1\ub^{\alpha-2}\ii(u-\ub)^2\nm\\
&\le&-\frac{\ub}{2M_u^2}\ii |\nn u|^2+\frac{\chi^2\ub}{2}\ii|\nn w|^2-b_1\ub^{\alpha-2}\ii(u-\ub)^2,\label{5.9}
\ees
where we have used Young's inequality, \eqref{5.1} and a basic inequality (cf. \cite[(4.14)]{DingWZZ-JDE2020}). Similarly,
\bes
&&\dv\ii\kk(v-\vb-\vb\ln\frac{v}{\vb}\rr)\nm\\
&\le&-\frac{\vb}{2}\ii \frac{|\nn v|^2}{v^2}+\frac{\xi^2\vb}{2}\ii|\nn u|^2-b_2\vb^{\beta-2}\ii(v-\vb)^2\nm\\
&\le&\frac{\xi^2\vb}{2}\ii|\nn u|^2-b_2\vb^{\beta-2}\ii(v-\vb)^2.\label{5.10}
\ees
Multiplying \eqref{5.10} by $\frac{\ub}{\xi^2\vb M_u^2}$ and adding the obtained result to \eqref{5.9}, we get
\bes
&&\dv\kk\{\ii\kk(u-\ub-\ub\ln\frac{u}{\ub}\rr)+\frac{\ub}{\xi^2\vb M_u^2}\ii\kk(v-\vb-\vb\ln\frac{v}{\vb}\rr)\rr\}\nm\\
&\le&\frac{\chi^2\ub}{2}\ii|\nn w|^2-b_1\ub^{\alpha-2}\ii(u-\ub)^2-\frac{b_2\ub\vb^{\beta-3}}{\xi^2M_u^2}\ii(v-\vb)^2.\nm
\ees
This shows \eqref{5.8}.
\end{proof}

\subsection{Global stability of $(\ub,\vb,0)$}

With fast decaying property \eqref{5.1a}, the large time behavior of the solution reads as follows.
\begin{lem}\label{l5.3}
Suppose that \eqref{1.2} and \eqref{5.1a} hold as well as
\bess
\alpha,\ \beta>\frac{N}{2}\ when\ N>2;\ \  \alpha\ge2,\ \beta>2\ for\ N=2.
\eess
Then,
\bes
\|u(\cdot,t)-\ub\|_{L^\yy(\oo)}+\|v(\cdot,t)-\vb\|_{L^\yy(\oo)}+\|w(\cdot,t)\|_{L^\yy(\oo)}\rightarrow0\ \ as\ \ t\rightarrow\yy.\label{5.11}
\ees
\end{lem}
\begin{proof}
Using \eqref{2.2}, it is easy to get
\bes
\frac12\dv\ii w^2&=&-\ii|\nn w|^2-\ii(u+v)w^2-\mu\ii w^2+\ii rw\nm\\
&\le&-\ii|\nn w|^2-\mu\ii w^2+M\ii r\ \ \ \ \ \ \faa\ t\in(0,\yy).\label{5.12}
\ees
Multiplying \eqref{5.12} by $\frac{\chi^2\ub}{2}$ and adding the obtained result to \eqref{5.8} to get
\bes
&&\dv\kk\{\ii\kk(u-\ub-\ub\ln\frac{u}{\ub}\rr)+\frac{\ub}{\xi^2\vb M_u^2}\ii\kk(v-\vb-\vb\ln\frac{v}{\vb}\rr)+\frac{\chi^2\ub}{4}\ii w^2\rr\}\nm\\
&\le&-b_1\ub^{\alpha-2}\ii(u-\ub)^2-\frac{b_2\ub\vb^{\beta-3}}{\xi^2M_u^2}\ii(v-\vb)^2-\frac{\chi^2\ub\mu}{2}\ii w^2+\frac{\chi^2\ub M}{2}\ii r\label{5.13}
\ees
for all $t\in(0,\yy)$.
Let
\bess
\mathcal{E}(t):=\ii\kk(u-\ub-\ub\ln\frac{u}{\ub}\rr)+\frac{\ub}{\xi^2\vb M_u^2}\ii\kk(v-\vb-\vb\ln\frac{v}{\vb}\rr)+\frac{\chi^2\ub}{4}\ii w^2,
\eess
and
\bess
\mathcal{F}(t):=b_1\ub^{\alpha-2}\ii(u-\ub)^2+\frac{b_2\ub\vb^{\beta-3}}{\xi^2M_u^2}\ii(v-\vb)^2+\frac{\chi^2\ub\mu}{2}\ii w^2.
\eess
Then it follows from \eqref{5.13} that
\bes
\mathcal{E}'(t)\le -\mathcal{F}(t)+\frac{\chi^2\ub M}{2}\ii r\ \ \ \ \ \ \faa\ t\in(0,\yy).\label{5.14}
\ees
Noting that $\mathcal{E}(t)\ge0$ for any $t>0$, integrating \eqref{5.14} over $(1,\yy)$ and using \eqref{5.1a}, we have
\bess
\int_1^\yy\mathcal{F}(t){\rm d}t\le \mathcal{E}(1)+\frac{\chi^2\ub M}{2}\int_1^\yy\ii r<\yy.
\eess
By the regularity of $(u,v,w)$ in \eqref{3.21a}, we know that $\mathcal{F}(t)$ is uniformly continuous in $[1,\yy)$. Making use of Barbalat's Lemma (cf. \cite{baiw-2016,Barbalat1959}), it follows that $\mathcal{F}(t)\rightarrow0$ as $t\rightarrow0$, i.e.,
\bes
\ii(u(\cdot,t)-\ub)^2+\ii(v(\cdot,t)-\vb)^2+\ii w^2\rightarrow0\ \ {\rm as}\ \ t\rightarrow0.\label{5.15}
\ees
Again by the regularities in \eqref{3.21a}, the $W^{1,\yy}(\oo)$-norm of $(u,v,w)$ is bounded. Hence, we can apply the Gagliardo-Nirenberg inequality to get
\bess
\|u-\ub\|_{L^\yy(\oo)}&\le& C_1\|u-\ub\|_{W^{1,\yy}(\oo)}^{N/(N+2)}\|u-\ub\|_{L^2(\oo)}^{2/(N+2)}\nm\\
&\le& C_2\|u-\ub\|_{L^2(\oo)}^{2/(N+2)}\ \ \ \ \ \ \faa\ t\in(0,\yy)
\eess
for some $C_1,C_2>0$. This combined with \eqref{5.15} gives the $L^\yy$ convergence statement of $u$ in \eqref{5.11}. By the same arguments, one can get the convergence statements of $v$ and $w$. The proof is finished.
\end{proof}

\begin{lem}\label{l5.3a}
Let $\alpha,\ \beta>\frac{N}{2}+1$ for $N>2$ and $\alpha\ge2,\ \beta>2$ for $N=2$. Suppose that $r$ satisfies \eqref{1.2} and
\bes
\ii r(\cdot,t)\le K e^{-\lm t}\ \ \ \ \ for\ all\ t>0 \label{5.15aa}
\ees
for some $K,\lm>0$. Then, there exist $\kappa,C>0$ such that
\bes
\|u(\cdot,t)-\ub\|_{L^\yy(\oo)}+\|v(\cdot,t)-\vb\|_{L^\yy(\oo)}+\|w(\cdot,t)\|_{L^\yy(\oo)}\le C e^{-\kappa t}\ \ \ \ \ for\ all\ t>0.\label{5.15ab}
\ees
\end{lem}
\begin{proof}
The following formula for $f(x)=x-a\ln x$ with $a>0$ is a consequence deduced by using L'H\^{o}pital's rule
 \[\lim_{x\to a} \frac{f(x)-f(a)}{(x-a)^2}=\lim_{x\to a} \frac{f'(x)}{2(x-a)}=\frac{1}{2a}.\]
Recalling \eqref{5.11}, there is $t_0>0$ such that, for all $t>t_0$,
\bes
\frac{1}{4\ub}\ii(u-\ub)^2&\le&\ii\kk( u-\ub-\ub\ln \frac{u}{\ub}\rr)\le\frac{1}{\ub}\ii(u-\ub)^2,\label{5.15ac}\\
\frac{1}{4\ub}\ii(v-\vb)^2&\le&\ii\kk(v-\vb-\vb\ln \frac{v}{\vb}\rr)\le\frac{1}{\vb}\ii(v-\vb)^2.\label{5.15ad}
\ees
Thanks to the right inequalities in \eqref{5.15ac} and \eqref{5.15ad}, one can find $C_1>0$ fulfilling
\bess
\mathcal{E}(t)\le \frac1{C_1}\mathcal{F}(t) \ \ \ \ \ \faa\ t>t_0,
\eess
where $\mathcal{E}(t)$ and $\mathcal{F}(t)$ are given in the proof of Lemma \ref{l5.3}. Plugging this into \eqref{5.14} gives
\bess
\mathcal{E}'(t)\le -\mathcal{F}(t)+\frac{\chi^2\ub M}{2}\ii r\le -C_1\mathcal{E}(t)+\frac{\chi^2\ub M}{2}\ii r\ \ \ \ \ \faa\ t>t_0
\eess
Hence, we have from this that
\bes
\mathcal{E}(t)&\le& e^{-C_1(t-t_0)}\mathcal{E}(t_0)+\frac{\chi^2\ub M}{2}e^{-C_1t}\int_{t_0}^t \kk(e^{C_1s}\ii r(x,s){\rm d}x\rr){\rm d}s\nm\\
&\le& \mathcal{E}(t_0)e^{C_1t_0}e^{-C_1t}+\frac{K\chi^2\ub M}{2}e^{-C_1t}\int_{t_0}^t e^{(C_1-\lm)s}{\rm d}s\nm\\
&\le& C_2e^{-\theta_1 t} \ \ \ \ \ \faa\ t>t_0\label{5.15ae}
\ees
for some $C_2,\theta_1>0$ where we used \eqref{5.15aa}.
According to the left inequalities in \eqref{5.15ac} and \eqref{5.15ad}, we use \eqref{5.15ae} to find $C_3>0$ fulfilling
\bess
\mathcal{F}(t)\le C_3e^{-\theta_1 t},
\eess
which implies
\bess
\int_\Omega (u-\ub)^2+\int_\Omega(v-\vb)^2+\int_\Omega w^2\le C_4 e^{-\theta_1 t} \ \ \ \ \ \faa\ t>t_0
\eess
for some $C_4>0$. Again by the uniform $W^{1,\yy}(\oo)$-boundedness of $(u,v,w)$ and the Gagliardo-Nirenberg inequality, we infer \eqref{5.15ab}.
\end{proof}

\begin{proof}[\rm\bf Proof of Theorem \ref{t1.3}]
We combine Lemma \ref{l5.3} and Lemma \ref{l5.3a} to get the desired conclusion.
\end{proof}

\subsection{Global stability of $(\ub,\vb,\wb)$}
\begin{lem}\label{l5.4}
Let $r$ be a positive constant and suppose that \eqref{5.6} holds. Then there exist $\tilde\chi,\tilde\xi,\tilde b_1,\tilde b_2>0$ such that, once either $\chi<\tilde\chi$ or $b_1>\tilde b_1$ and $b_2>\tilde b_2$ or $b_1>\tilde b_1$ and $\xi<\tilde \xi$, one has
\bes
\|u(\cdot,t)-\ub\|_{L^\yy(\oo)}+\|v(\cdot,t)-\vb\|_{L^\yy(\oo)}+\|w(\cdot,t)-\wb\|_{L^\yy(\oo)}\rightarrow0\ \ as\ \ t\rightarrow\yy.\label{5.16}
\ees
\end{lem}
\begin{proof}
First of all, we claim that
\bes
b_1\ub^{\alpha-2}-\frac{\chi^2\ub M^3}{2r}>0\ \ {\rm and}\ \ \frac{b_2\ub\vb^{\beta-3}}{\xi^2M_u^2}-\frac{\chi^2\ub M^3}{2r}>0\label{5.15a}
\ees
holds for small $\chi$ or large $b_1$ and $b_2$ or large $b_1$ and small $\xi$. Since the discussion of case $N=2$ is same with the case $N>2$, we only get the arguments of $N>2$. We recall from Lemma \ref{l5.0}(i) that
\bess
M_u:=C_1(\chi^\theta+1)(b_1^{-\frac{1}{\alpha-1}}+1)\kk(b_1^{-1}(b_1^{-\frac{1}{\alpha-1}}+1)+b_2^{-1}(b_2^{-\frac{1}{\beta-1}}+1)+1\rr)^\theta
\eess
for some $C_1>0,\theta>1$ independent of $\chi,\xi,b_1,b_2$. Inserting this into \eqref{5.15a} and using the definitions of $\ub$ and $\vb$, one can find $C_2,C_3,C_4,C_5>0$ independent of $\chi,\xi,b_1,b_2$ such that \eqref{5.15a} can be written as
\bes
C_2b_1^{\frac1{\alpha-1}}-C_3\chi^2b_1^{-\frac1{\alpha-1}}>0\label{5.16a}
\ees
and
\bes
\frac{C_4b_1^{-\frac1{\alpha-1}}b_2^{\frac2{\beta-1}}}
{\xi^2(\chi^\theta+1)(b_1^{-\frac{1}{\alpha-1}}+1)\kk(b_1^{-1}(b_1^{-\frac{1}{\alpha-1}}+1)+b_2^{-1}(b_2^{-\frac{1}{\beta-1}}+1)+1\rr)^\theta}
-C_5\chi^2b_1^{-\frac1{\alpha-1}}>0.     \label{5.16b}
\ees
The inequalities \eqref{5.16a} and \eqref{5.16b} can be satisfied in the following three cases:
\begin{itemize}
\item {\bf Small $\chi$}: There is $\tilde\chi=\tilde\chi(\xi,b_1,b_2)>0$ such that \eqref{5.16a} and \eqref{5.16b} hold for $\chi<\tilde\chi$.
\item {\bf Large $b_1$ and $b_2$}: There exist $\tilde b_1=\tilde b_1(\chi)>0$ and $\tilde b_2=\tilde b_2(\chi,\xi,b_1)>0$ such that \eqref{5.16a} and \eqref{5.16b} hold for $b_1>\tilde b_1$ and $b_2>\tilde b_2$.
\item {\bf Large $b_1$ and small $\xi$}: One can find $\tilde b_1=\tilde b_1(\chi)>0$ and $\tilde\xi>0$ depending on $\chi,b_1,b_2$ such that \eqref{5.16a} and \eqref{5.16b} hold for $b_1>\tilde b_1$ and $\xi<\tilde \xi$.
\end{itemize}
Therefore, \eqref{5.15a} holds under the conditions of this lemma.

By honest computations, we have
\bes
&&\dv\ii\kk(w-\wb-\wb\ln\frac{w}{\wb}\rr)\nm\\
&=&-\wb\ii\frac{|\nn w|^2}{w^2}+\ii(w-\wb)\kk(-(u+v)-\mu+\frac{r}{w}\rr)\nm\\
&\le&-\frac{\wb}{M^2}\ii|\nn w|^2+\ii(w-\wb)\kk[(\ub-u)+(\vb-v)+\frac{r(\wb-w)}{\wb w}\rr]\nm\\
&\le&-\frac{\wb}{M^2}\ii|\nn w|^2-\ii(u-\ub)(w-\wb)-\ii(v-\vb)(w-\wb)-\ii\frac{r}{\wb w}(w-\wb)^2\nm\\
&\le&-\frac{\wb}{M^2}\ii|\nn w|^2-\ii(u-\ub)(w-\wb)-\ii(v-\vb)(w-\wb)-\frac{r}{\wb M}\ii(w-\wb)^2\nm\\
&\le&-\frac{\wb}{M^2}\ii|\nn w|^2+\frac{\wb M}{r}\ii(u-\ub)^2+\frac{\wb M}{r}\ii(v-\vb)^2-\frac{r}{2\wb M}\ii(w-\wb)^2\label{5.17}
\ees
for all $t\in(0,\yy)$. We multiply \eqref{5.17} with $\frac{\chi^2\ub M^2}{2\wb}$ and add the obtained result to \eqref{5.8} to eliminate $\ii|\nn w|^2$, i.e.,
\bes
&&\dv\kk\{\ii\kk(u-\ub-\ub\ln\frac{u}{\ub}\rr)+\frac{\ub}{\xi^2\vb M_u^2}\ii\kk(v-\vb-\vb\ln\frac{v}{\vb}\rr)
+\frac{\chi^2\ub M^2}{2\wb}\ii\kk(w-\wb-\wb\ln\frac{w}{\wb}\rr)\rr\}\nm\\
&\le&-\kk(b_1\ub^{\alpha-2}-\frac{\chi^2\ub M^3}{2r}\rr)\ii(u-\ub)^2-\kk(\frac{b_2\ub\vb^{\beta-3}}{\xi^2M_u^2}-\frac{\chi^2\ub M^3}{2r}\rr)\ii(v-\vb)^2\nm\\
&&-\frac{r\chi^2\ub M}{4\wb^2}\ii(w-\wb)^2\ \ \ \ \ \ \faa\ t\in(0,\yy),\label{5.18}
\ees
where the coefficients in the right side are all positive according to \eqref{5.15a}. By letting
\bess
\mathcal{E}_1(t)=\ii\kk(u-\ub-\ub\ln\frac{u}{\ub}\rr)+\frac{\ub}{\xi^2\vb M_u^2}\ii\kk(v-\vb-\vb\ln\frac{v}{\vb}\rr)
+\frac{\chi^2\ub M^2}{2\wb}\ii\kk(w-\wb-\wb\ln\frac{w}{\wb}\rr)
\eess
and
\bess
\mathcal{F}_1(t)&=&\kk(b_1\ub^{\alpha-2}-\frac{\chi^2\ub M^3}{2r}\rr)\ii(u-\ub)^2+\kk(\frac{b_2\ub\vb^{\beta-3}}{\xi^2M_u^2}-\frac{\chi^2\ub M^3}{2r}\rr)\ii(v-\vb)^2\nm\\
&&+\frac{r\chi^2\ub M}{4\wb^2}\ii(w-\wb)^2,
\eess
we have from \eqref{5.18} that
\bes
\mathcal{E}_1'(t)\le -\mathcal{F}_1(t)\ \ \ \ \ \ \faa\ t\in(0,\yy).\label{5.19}
\ees
Clearly, $\mathcal{E}_1(t),\mathcal{F}_1(t)\ge0$ for any $t>0$. Parallel to the arguments performed in Lemma \ref{l5.3}, one can show \eqref{5.16}.
\end{proof}

\begin{proof}[\rm\bf Proof of Theorem \ref{t1.4}]
With \eqref{5.16} and \eqref{5.19} at hand, similar to the arguments in Lemma \ref{l5.3a} (cf. \cite{baiw-2016,WW-JDDE2020,WW-ZAMP2018}), one can easily show the conclusion.
\end{proof}

 \end{document}